\documentclass[12pt,twoside]{amsart}
\usepackage{amsmath}
\usepackage{amssymb}
\usepackage{amsthm}
\usepackage{latexsym}
\usepackage{amscd}
\usepackage{xypic}
\xyoption{curve}
\usepackage{ifthen} 
\usepackage{hyperref} 
\usepackage{graphicx}
\usepackage{enumerate}
\usepackage{enumitem}

 \def\cocoa{{\hbox{\rm C\kern-.13em o\kern-.07em C\kern-.13em o\kern-.15em A}}}

\usepackage{xcolor}

\newtheorem{theorem}{Theorem}[section]

\newtheorem{proposition}[theorem]{Proposition}
\newtheorem{corollary}[theorem]{Corollary}

\theoremstyle{definition}
\newtheorem{remark}[theorem]{Remark}

\newcommand {\Hom}{\mathcal{H}\kern -0.25ex{\mathit om}}
\newcommand {\Ext}{\mathcal{E}\kern -0.25ex{\mathit xt}}

\newcommand {\rk}{\mathrm{rk}}

\newcommand {\Hilb}{\mathcal{H}\kern -0.25ex{\mathit ilb\/}}

\newcommand {\cK}{\mathcal{K}}
\newcommand {\cA}{\mathcal{A}}
\newcommand {\cB}{\mathcal{B}}

\newcommand {\bZ}{\mathbb{Z}}

\newcommand {\bP}{\mathbb{P}}

\newcommand{\cU}{{\mathcal U}}

\newcommand{\cE}{{\mathcal E}}
\newcommand{\cF}{{\mathcal F}}
\newcommand{\cM}{{\mathcal M}}

\newcommand{\cO}{{\mathcal O}}
\newcommand{\cG}{{\mathcal G}}
\newcommand{\cP}{{\mathcal P}}
\newcommand{\cI}{{\mathcal I}}

\newcommand{\Pic}{\operatorname{Pic}}

\def\p#1{{\bP^{#1}}}

\def\ga#1{{{\accent"12 #1}}}
 
\title[Ulrich bundles on non--special surfaces with $p_g=0$ and $q=1$]{Ulrich bundles on non--special surfaces \\ with $p_g=0$ and $q=1$}



\author[Gianfranco Casnati]{Gianfranco Casnati}
\thanks{The author is a member of GNSAGA group of INdAM and is supported by the framework of PRIN 2015 \lq Geometry of Algebraic Varieties\rq, cofinanced by MIUR}

\begin{document}

  \begin{abstract}
\noindent Let $S$ be a surface with $p_g(S)=0$, $q(S)=1$ and endowed with a very ample line bundle $\mathcal O_S(h)$ such that $h^1\big(S,\mathcal O_S(h)\big)=0$. We show that such an $S$ supports families of dimension $p$ of pairwise non--isomorphic, indecomposable, Ulrich bundles for arbitrary large $p$. Moreover, we show that $S$ supports stable Ulrich bundles of rank $2$ if the genus of the general element in $\vert h\vert$ is at least $2$. 
  \end{abstract}

\maketitle

\section{Introduction and Notation}
Throughout the whole paper we will work on an uncountable algebraically closed field $k$ of characteristic $0$ and $\p N$ will denote the projective space over $k$ of dimension $N$. The word surface will always denote a projective smooth connected surface.

If $X$ is a smooth variety, then the study of vector bundles supported on $X$ is an important tool for understanding its geometric properties. If $X\subseteq\p N$, then $X$ is naturally polarised by the very ample line bundle $\cO_X(h):=\cO_{\p N}(1)\otimes\cO_X$: in this case, at least from a cohomological point of view, the simplest bundles $\cF$ on $X$ are the ones which are {\sl Ulrich with respect to $\cO_X(h)$}, i.e. such that 
$$
h^i\big(X,\cF(-ih)\big)=h^j\big(X,\cF(-(j+1)h)\big)=0
$$
for each $i>0$ and $j<\dim(X)$. 

The existence of Ulrich bundles on each variety is a problem raised by D. Eisenbud and F.O. Schreyer in \cite{E--S--W} (see \cite{Bea3} for a survey on Ulrich bundles). There are many partial results (e.g. see \cite{A--C--MR}, \cite{A--F--O}, \cite{Bea}, \cite{Bea1}, \cite{Bea2}, \cite{B--N},  \cite{C--H1}, \cite{C--H2}, \cite{C--G}, \cite{C--K--M1}, \cite{C--K--M2}, \cite{MR}, \cite{MR--PL1}, \cite{MR--PL2}, \cite{PL--T}). Nevertheless, all such results and those ones proved in \cite{F--PL} seem to suggest that Ulrich bundles exist at least when $X$ satisfies an extra technical condition, namely that $X$ is {\sl arithmetically Cohen--Macaulay}, i.e. projectively normal and such that
$$
h^i\big(X,\cO_S(th)\big)=0
$$
for each $i=1,\dots,\dim(X)-1$ and $t\in\bZ$. When $X$ is not arithmetically Cohen--Macaulay, the literature is very limited (e.g. see \cite{Bea2} and \cite{Cs}).

Now let $S\subseteq\p N$ be a surface and set $p_g(S):=h^2\big(S,\cO_S\big)$, $q(S):=h^1\big(S,\cO_S\big)$, whence $\chi(\cO_S):=1-q(S)+p_g(S)=0$.  Thanks to the Enriques--Kodaira classification of surfaces, we know that $\kappa(S)\le1$ and $K_S^2\le0$ (see \cite{BeaBook}, Theorem X.4 and Lemma VI.1). In what follows we will denote by $\Pic(S)$  the Picard group of $S$: it is a group scheme and the connected component $\Pic^0(S)\subseteq\Pic(S)$  of the identity is an abelian variety of dimension $q(S)$ parameterizing the line bundles algebraically equivalent to $\cO_S$. 

In this paper we first rewrite the proof of Proposition 6 of  \cite{Bea3}, in order to be able to extend its statement to a slightly wider class of surfaces. 

Our modified statement  (which holds also without the hypothesis that $k$ is uncountable) is as follows: recall that $\cO_S(h)$ is called {\sl special} if $h^1\big(S,\cO_S(h)\big)\ne0$, {\sl non--special} otherwise. 

\begin{theorem}
\label{tExistence}
Let $S$ be a surface with $p_g(S)=0$, $q(S)=1$ and endowed with a very ample non--special line bundle $\cO_S(h)$.

If $\cO_S(\eta)\in\Pic^0(S)\setminus\{\ \cO_S\ \}$ is such that $h^0\big(S,\cO_S(K_S\pm\eta)\big)=h^1\big(S,\cO_S(h\pm\eta)\big)=0$, then for each general $C\in\vert \cO_S(h)\vert$ and each general set $Z\subseteq C$ of $h^0\big(S,\cO_S(h)\big)$ points, there is a rank $2$ Ulrich bundle $\cE$ with respect to $\cO_S(h)$ fitting into the exact sequence
\begin{equation}
\label{seqUlrich}
0\longrightarrow\cO_S(h+K_S+\eta)\longrightarrow\cE\longrightarrow\cI_{Z\vert S}(2h+\eta)\longrightarrow0.
\end{equation}
\end{theorem}

As pointed out in \cite{Bea3}, Proposition 6, when $S$ is a bielliptic surface then each very ample line bundle $\cO_S(h)$ is automatically non--special and there always exists a non--trivial $\cO_S(\eta)\in\Pic^0(S)$ of order $2$ satisfying the above vanishings: thus the  bundle $\cE$ defined in Theorem \ref{tExistence} is actually {\sl special}, i.e. $c_1(\cE)=3h+K_S$. We can argue similarly if $S$ is either {\sl anticanonical}, i.e. $\vert-K_S\vert\ne\emptyset$, or geometrically ruled.

A condition forcing the indecomposability of a coherent sheaf $\cF$ on an $n$--dimensional variety $X$ is its stability. Recall that the {\sl slope} $\mu(\cF)$ and the {\sl reduced Hilbert polynomial} $p_{\cF}(t)$ of $\cF$ with respect to the very ample polarisation $\cO_X(h)$ are 
$$
\mu(\cF)= c_1(\cF)h^{n-1}/\rk(\cF), \qquad p_{\cF}(t)=\chi(\cF(th))/\rk(\cF).
$$
The coherent sheaf $\cF$ is called $\mu$--semistable (resp. $\mu$--stable) if for all subsheaves $\mathcal G$ with $0<\rk(\mathcal G)<\rk(\cF)$ we have $\mu(\mathcal G) \le \mu(\cF)$ (resp. $\mu(\mathcal G)< \mu(\cF)$).

The coherent sheaf $\cF$ is called semistable (resp. stable) if for all $\mathcal G$ as above $p_{\mathcal G}(t) \le  p_{\cF}(t)$ (resp. $p_{\mathcal G}(t) <  p_{\cF}(t)$) for $t\gg0$. 

On an arbitrary variety we have the following chain of implications
$$
\text{$\cF$ is $\mu$--stable}\Rightarrow\text{$\cF$ is stable}\Rightarrow\text{$\cF$ is semistable}\Rightarrow\text{$\cF$ is $\mu$--semistable.}
$$
Nevertheless, when we restrict our attention to Ulrich bundles, the two notions of (semi)stability and $\mu$--(semi)stability actually coincide.

A priori, it is not clear whether the bundles constructed in Theorem \ref{tExistence} are stable. In Section \ref{sStability} we deal with their stability as follows.
The {\sl sectional genus} of $S$ with respect to $\cO_S(h)$ is defined as the genus of a general element of $\vert h\vert$. By the adjunction formula 
$$
\pi(\cO_S(h)):=\frac{h^2+hK_S}2+1.
$$
Notice that the equality $\pi(\cO_S(h))=0$ would imply the rationality of $S$ (e.g. see \cite{A--S} and the references therein), contradicting $q(S)=1$. Thus $\pi(\cO_S(h))\ge1$ in our setup. 

\begin{theorem}
\label{tStable}
Let $S$ be a surface with $p_g(S)=0$, $q(S)=1$ and endowed with a very ample non--special line bundle $\cO_S(h)$.

If $\pi(\cO_S(h))\ge2$, then the bundle $\cE$ constructed in Theorem \ref{tExistence} from a very general set $Z\subseteq C\subseteq S$ of $h^0\big(S,\cO_S(h)\big)$ points is stable.
\end{theorem}

Once that the existence of Ulrich bundles of low rank is proved, one could be interested in understanding how large  a family of Ulrich bundles supported on $S$ can actually be. In particular we say that a smooth variety $X\subseteq\p N$ is  {\sl Ulrich--wild} if it supports families of dimension $p$ of pairwise non--isomorphic, indecomposable, Ulrich bundles for arbitrary large $p$. 

The last result proved in this paper concerns the Ulrich--wildness of the surfaces we are dealing with.

\begin{theorem}
\label{tWild}
Let $S$ be a surface with $p_g(S)=0$, $q(S)=1$ and endowed with a very ample non--special line bundle $\cO_S(h)$.
Then $S$ is Ulrich--wild.
\end{theorem}

In Section \ref{sGeneral} we list some general results on Ulrich bundles on polarised surfaces. In Section \ref{sExistence} we prove Theorem \ref{tExistence}. In Section \ref{sStability} we first recall some easy facts about the stability of Ulrich bundles, giving finally the proof of Theorem \ref{tStable}. In Section \ref{sWild} we prove Theorem \ref{tWild}.

Finally, the author would like to thank the referee for her/his comments which have allowed us to improve the whole exposition.

\section{General results}
\label{sGeneral}
In general, an Ulrich bundle $\cF$ on $X\subseteq\p N$ collects many interesting properties (see Section 2 of \cite{E--S--W}). The following ones are particularly important.
\begin{itemize}
\item $\cF$ is globally generated and its direct summands are Ulrich as well. 
\item $\cF$ is {\sl initialized}, i.e. $h^0\big(X,\cF(-h)\big)=0$ and $h^0\big(X,\cF\big)\ne0$.
\item $\cF$ is {\sl aCM}, i.e. $h^i\big(X,\cF(th)\big)=0$ for each $i=1,\dots,\dim(X)-1$ and $t\in \bZ$.
\end{itemize}

Let $S$ be a surface. The Serre duality for $\cF$ is
$$
h^i\big(S,\cF\big)=h^{2-i}\big(S,\cF^\vee(K_S)\big),\qquad i=0,1,2,
$$
and the Riemann--Roch theorem is
\begin{equation}
\label{RRGeneral}
\begin{aligned}
h^0\big(S,\cF\big)&+h^{2}\big(S,\cF\big)=h^{1}\big(S,\cF\big)+\\
&+\rk(\cF)\chi(\cO_S)+\frac{c_1(\cF)(c_1(\cF)-K_S)}2-c_2(\cF).
\end{aligned}
\end{equation}

\begin{proposition}
\label{pUlrich}
Let $S$ be a surface endowed with a very ample line bundle  $\cO_S(h)$.

If $\cE$ is a vector bundle on $S$, then the following assertions are equivalent:
\begin{enumerate}
\item $\cE$ is an Ulrich bundle with respect to $\cO_S(h)$;
\item $\cE^\vee(3h+K_S)$ is an Ulrich bundle with respect to $\cO_S(h)$;
\item $\cE$ is an aCM bundle and 
\begin{equation}
\label{eqUlrich}
\begin{gathered}
c_1(\cE)h=\rk(\cE)\frac{3h^2+hK_S}2,\\ 
c_2(\cE)=\frac{c_1(\cE)^2-c_1(\cE)K_S}2-\rk(\cE)(h^2-\chi(\cO_S));
\end{gathered}
\end{equation}
\item $h^0\big(S,\cE(-h)\big)=h^0\big(S,\cE^\vee(2h+K_S)\big)=0$ and Equalities \eqref{eqUlrich} hold.
\end{enumerate}
\end{proposition}
\begin{proof}
See \cite{Cs}, Proposition 2.1.
\end{proof}

The following corollaries are immediate consequences of the above characterization.

\begin{corollary}
\label{cUlrichLine}
Let $S$ be a surface endowed with a very ample line bundle  $\cO_S(h)$.

If $\cO_S(D)$ is a line bundle on $S$, then the following assertions are equivalent:
\begin{enumerate}
\item $\cO_S(D)$ is an Ulrich bundle with respect to $\cO_S(h)$;
\item $\cO_S(3h+K_S-D)$ is an Ulrich bundle with respect to $\cO_S(h)$;
\item $\cO_S(D)$ is an aCM bundle and 
\begin{equation}
\label{eqLineBundle}
D^2=2(h^2-\chi(\cO_S))+DK_S,\qquad Dh=\frac12(3h^2+hK_S);
\end{equation}
\item $h^0\big(S,\cO_S(D-h)\big)=h^0\big(S,\cO_S(2h+K_S-D)\big)=0$ and Equalities \eqref{eqLineBundle} hold.
\end{enumerate}
\end{corollary}
\begin{proof}
See \cite{Cs}, Corollary 2.2.
\end{proof}

\section{Existence of rank $2$ Ulrich bundles}
\label{sExistence}

We start this section by recalling that if $S$ is any surface, then the connected component $\Pic^0(S)$ of the identity inside $\Pic(S)$ is an abelian variety of dimension $q(S)$ called {\sl Picard variety of $S$}. The quotient is a finitely generated abelian group called {\sl  N\'eron--Severi group of $S$}.

Now, let $S$ be a surface with $p_g(S)=0$ and $q(S)=1$. Then $\Pic^0(S)$ is an elliptic curve: in particular $\Pic^0(S)$ contains three pairwise distinct non--trivial divisors of order $2$.

In order to prove Theorem \ref{tExistence} we will make use of the Hartshorne--Serre correspondence on surfaces. We recall that a locally complete intersection subscheme $Z$ of dimension zero on a surface $S$ is Cayley--Bacharach (CB for short) with respect to a line bundle  $\cO_S(A)$ if, for each $Z'\subseteq Z$ of degree $\deg(Z)-1$, the natural morphism $H^0\big(S,\cI_{Z\vert S}(A)\big)\to H^0\big(S,\cI_{Z'\vert S}(A)\big)$ is an isomorphism.

\begin{theorem}
\label{tCB}
Let $S$ be a surface and $Z\subseteq S$ a locally complete intersection subscheme of dimension $0$. 
 
Then there exists a vector bundle $\cF$ of rank $2$ on $S$ fitting into an exact sequence of the form
\begin{equation}
\label{seqCB}
0\longrightarrow \cO_S\longrightarrow \cF\longrightarrow \cI_{Z\vert S}(A)\longrightarrow0,
\end{equation}
if and only if $Z$ is CB with respect to $\cO_S(A+K_S)$. 
\end{theorem}
\begin{proof}
See Theorem 5.1.1 in \cite{H--L}.
\end{proof}

We now prove Theorem \ref{tExistence} stated in the introduction. As we already noticed therein, its proof for $hK_S=0$ coincides with the one of Proposition 6 in \cite{Bea3} because in this case the vanishing $h^1\big(S,\cO_S(h\pm\eta)\big)=0$ follows immediately from the Kodaira vanishing theorem as we will show below in Corollary \ref{cBielliptic}. 

\medbreak
\noindent{\it Proof of Theorem \ref{tExistence}.}
Recall that by hypothesis $p_g(S)=h^1\big(S,\cO_S(h)\big)=0$ and $q(S)=1$. It follows that  $\chi(\cO_S)=0$ and
\begin{equation*}
\label{eqh^2}
h^2\big(S,\cO_S(h)\big)=h^0\big(S,\cO_S(K_S-h)\big)\le h^0\big(S,\cO_S(K_S)\big)=0,
\end{equation*}
thus $S\subseteq\p N$, where 
\begin{equation}
\label{eqDimension}
N:=h^0\big(S,\cO_S(h)\big)-1=\frac{h^2-hK_S}2-1\ge4,
\end{equation}
because $q(S)=0$ for each surface $S\subseteq\p3$.

Let $C:=S\cap H\in\vert h\vert$ be a general hyperplane section and let $i\colon C\to S$ be the inclusion morphism. 
The curve $C$ is non--degenerate in $\p{N-1}\cong H\subseteq\p N$. Indeed the exact sequence
$$
0\longrightarrow \cI_{S\vert\p N}(1)\longrightarrow \cO_{\p N}(1)\longrightarrow \cO_{S}(h)\longrightarrow 0
$$
implies $h^0\big(\p N,\cI_{S\vert\p N}(1)\big)=h^1\big(\p N,\cI_{S\vert\p N}(1)\big)=0$. Thus, the exact sequence
$$
0\longrightarrow \cI_{S\vert\p N}(1)\longrightarrow  \cI_{C\vert\p N}(1)\longrightarrow  \cI_{C\vert S}(h)\longrightarrow 0
$$
implies $h^0\big(\p N,\cI_{C\vert\p N}(1)\big)=1$, because $\cI_{C\vert S}(h)\cong\cO_S$. Finally the exact sequence 
$$
0\longrightarrow \cI_{H\vert\p N}(1)\longrightarrow  \cI_{C\vert\p N}(1)\longrightarrow  \cI_{C\vert H}(1)\longrightarrow 0
$$
and the isomorphism $\cI_{H\vert\p N}(1)\cong\cO_{\p N}$ yields $h^0\big(C,\cI_{C\vert H}(1)\big)=0$. 

It follows the existence of a reduced subscheme $Z\subseteq C\subseteq S$ of degree $N+1$ whose points are in general position inside $H\cong\p{N-1}$. Thus $Z$ is CB with respect to $\cO_S(h)$, hence there exists Sequence \eqref{seqCB} with $\cO_S(A)\cong\cO_S(h-K_S)$, thanks to Theorem \ref{tCB}.

Let $\cO_S(\eta)\in\Pic^0(S)\setminus\{\ \cO_S\ \}$ be such that $h^0\big(S,\cO_S(K_S\pm\eta)\big)=h^1\big(S,\cO_S(h\pm\eta)\big)=0$ and set $\cE:=\cF(h+K_S+\eta)$. The bundle $\cE$ fits into Sequence \eqref{seqUlrich} and satisfies Equalities \eqref{eqUlrich}. If we show that $h^0\big(S,\cE(-h)\big)=h^0\big(S,\cE^\vee(2h+K_S)\big)=0$, then we conclude that $\cE$ is Ulrich thanks to Proposition \ref{pUlrich} above. Notice that the second vanishing is equivalent to $h^0\big(S,\cE(-h-2\eta)\big)=0$ because $c_1(\cE)=3h+K_S+2\eta$.

The vanishing $h^0\big(S,\cO_S(K_S\pm\eta)\big)=0$ implies
\begin{equation*}
\begin{gathered}
\label{Bound}
h^0\big(S,\cE(-h)\big)\le h^0\big(S, \cI_{Z\vert S}(h+\eta)\big),\qquad
h^0\big(S,\cE(-h-2\eta)\big)\le h^0\big(S, \cI_{Z\vert S}(h-\eta)\big).
\end{gathered}
\end{equation*}
The exact sequence 
\begin{equation}
\label{seqIdeals}
0\longrightarrow \cI_{C\vert S}\longrightarrow  \cI_{Z\vert S}\longrightarrow  \cI_{Z\vert C}\longrightarrow 0
\end{equation}
and the isomorphisms $\cI_{C\vert S}\cong\cO_S(-h)$ and $\cI_{Z\vert C}\cong\cO_C(-Z)$ imply
$$
h^0\big(S, \cI_{Z\vert S}(h\pm \eta)\big)\le h^0\big(C, \cO_C(-Z)\otimes\cO_S(h\pm\eta)\big)
$$
because $h^0\big(S, \cO_S(\pm\eta)\big)=0$. Thanks to the general choice of the points in $Z$,
the Riemann--Roch theorem on $C$ and the adjunction formula $\cO_C(K_C)\cong i^*\cO_S(h+K_S)$ on $S$ give
\begin{align*}
h^0\big(&C, \cO_C(-Z)\otimes\cO_S(h\pm\eta)\big)=h^0\big(C,  i^*\cO_S(h\pm\eta)\big)-\deg(Z)=\\
&=h^2+1-\pi(\cO_S(h))-\deg(Z)+h^1\big(C,  i^*\cO_S(h\pm\eta)\big)=h^0\big(C,  i^*\cO_S(K_S\mp\eta)\big).
\end{align*}
The exact sequence
\begin{equation}
\label{seqHyperplane}
0\longrightarrow \cO_S(-h)\longrightarrow  \cO_S\longrightarrow  \cO_C\longrightarrow 0
\end{equation}
implies the existence of the  exact sequence
\begin{align*}
H^0\big(S,\cO_S(K_S\mp\eta)\big)&\longrightarrow  H^0\big(C, i^*\cO_S(K_S\mp\eta)\big)\longrightarrow\\
&\longrightarrow H^1\big(S,\cO_S(K_S-h\mp\eta)\big)\cong H^1\big(S,\cO_S(h\pm\eta)\big).
\end{align*}
Thus the hypothesis on $\cO_S(K_S\pm\eta)$ and $\cO_S(h\pm\eta)$ forces $h^0\big(C, i^*\cO_S(K_S\mp\eta)\big)=0$.
\qed
\medbreak

It is natural to ask when the vanishings $h^1\big(S,\cO_S(K_S\pm\eta)\big)=h^1\big(S,\cO_S(h\pm\eta)\big)=0$ actually occur. We list below some related result.

\begin{corollary}
Let $S$ be a surface with $p_g(S)=0$, $q(S)=1$ and endowed with a very ample non--special line bundle $\cO_S(h)$.

Then $S$ supports Ulrich bundles of rank $r\le2$.
\end{corollary}
\begin{proof}
Since each direct summand of an Ulrich bundle is Ulrich as well, it follows from Theorem \ref{tExistence} that it suffices to prove the existence of $\cO_S(\eta)\in\Pic^0(S)\setminus\{\ \cO_S\ \}$ such that $h^0\big(S,\cO_S(K_S\pm\eta)\big)=h^1\big(S,\cO_S(h\pm\eta)\big)=0$.

Let $\cP$ be the Poincar\'e line bundle on $S\times\Pic(S)$. Recall that (e.g. see \cite{Mu}, Lecture 19), if $p\colon S\times\Pic(S)\to\Pic(S)$ is the projection on the second factor and $\mathcal L\in \Pic(S)$, then the restriction of $\mathcal P$ to the fibre $p^{-1}(\mathcal L)\cong S$ is isomorphic to the line bundle $\mathcal L$. The line bundle $\mathcal P$ is thus flat on $\Pic(S)$.

Let $\mathcal P_0$ be the restriction of $\mathcal P$ to $S\times\Pic^0(S)$, $A\subseteq S$ a divisor, $s\colon S\times\Pic(S)\to S$ the projection on the first factor. The line bundle $\mathcal P_0\otimes s^*\cO_S(A)$ is flat over $\Pic^0(S)$ and parameterizes  the line bundles on $S$ algebraically equivalent to $\cO_S(A)$. Thus the semicontinuity theorem (e.g. see Theorem III.12.8 of \cite{Ha2}) applied to the sheaf $\mathcal P_0\otimes s^*\cO_S(A)$ and the map $p_0\colon S\times\Pic^0(S)\to\Pic^0(S)$ imply that for each $i=0,1,2$ and $c\in\bZ$ the sets
\begin{gather*}
\mathcal V^i_A(c):=\{\ \eta\in\Pic^0(S)\ \vert\ h^i\big(S,\cO_S(A\pm\eta)\big)>c\ \},
\end{gather*}
are closed inside $\Pic^0(S)$. In particular $\mathcal V:=\mathcal V^1_h(0)\cup \mathcal V^0_{K_S}(0)$ is closed.

By definition $\cO_S\in\Pic^0(S)\setminus\mathcal V\ne\emptyset$. Thus for each general $\cO_S(\eta)\in\Pic^0(S)$, the hypothesis $h^0\big(S,\cO_S(K_S\pm\eta)\big)=h^1\big(S,\cO_S(h\pm\eta)\big)=0$ is satisfied and the statement is then completely proved.
\end{proof}

Notice that the above result guarantees the existence of an Ulrich bundle $\cE$ with $c_1(\cE)=3h+K_S+2\eta$ fitting into Sequence \eqref{seqUlrich}. Such bundle is special if and only if  $\cO_S(\eta)$ has order $2$. It is not clear if such a choice can be done in general. Anyhow in some particular cases we can easily prove an existence result also for special Ulrich bundles: we start from Beauville's result for {\sl bielliptic surfaces}, i.e. minimal surfaces $S$ with $p_g(S)=0$, $q(S)=1$ and $\kappa(S)=0$ (see Proposition 6 of \cite{Bea3}).

\begin{corollary}
\label{cBielliptic}
Let $S$ be a bielliptic surface endowed with a very ample line bundle $\cO_S(h)$.

Then $\cO_S(h)$ is non--special and $S$ supports special Ulrich bundles of rank $2$.
\end{corollary}
\begin{proof}
If $\kappa(S)=0$, then $K_S$ is numerically trivial, hence $h-K_S\pm\eta$ is ample for each choice of $\cO_S(\eta)\in\Pic^0(S)$, thanks to the Nakai criterion. Thus the vanishing $h^1\big(S,\cO_S(h\pm\eta)\big)=0$ follows from the Kodaira vanishing theorem: in particular $\cO_S(h)$ is non--special. 

We can find $\cO_S(\eta)\in \Pic^0(S)\setminus\{\ \cO_S,\cO_S(\pm K_S)\ \}$ of order $2$, because there are three non--trivial and pairwise non--isomorphic elements of order $2$ in $\Pic^0(S)$. 
Thus $h^0\big(S,\cO_S(K_S\pm\eta)\big)=0$ because $K_S\pm\eta$ is not trivial by construction, hence the statement follows from Theorem \ref{tExistence}.
\end{proof} 

The surface $S$ is {\sl anticanonical} if $\vert-K_S\vert\ne\emptyset$: in particular $p_g(S)=0$. The ampleness of $\cO_S(h)$ implies $hK_S<0$ in this case. 

\begin{corollary}
\label{cAnticanonical}
Let $S$ be an anticanonical surface with $q(S)=1$ and endowed with a very ample line bundle $\cO_S(h)$.

Then $\cO_S(h)$ is non--special and $S$ supports special Ulrich bundles of rank $2$.
\end{corollary}
\begin{proof}
If $A\in\vert -K_S\vert$, then $\omega_A\cong\cO_A$ by the adjunction formula. We have $h^1\big(A,\cO_S(h\pm\eta)\otimes\cO_A\big)=h^0\big(A,\cO_S(-h\mp\eta)\otimes\cO_A\big)$, for each $\cO_S(\eta)\in\Pic^0(S)$. 

On the one hand, if $h^0\big(A,\cO_S(-h\mp\eta)\otimes\cO_A\big)>0$, then $-hC\ge0$ for some irreducible component $C\subseteq A$. On the other hand  $\cO_S(h)$ is ample, hence $hC>0$. 

The contradiction implies $h^0\big(A,\cO_S(-h\mp\eta)\otimes\cO_A\big)=0$, hence the cohomology of the exact sequence
$$
0\longrightarrow \cO_S(h+K_S\mp\eta)\longrightarrow \cO_S(h\mp\eta)\longrightarrow \cO_S(h\mp\eta)\otimes\cO_A\longrightarrow0
$$
and the Kodaira vanishing theorem yield $h^1\big(S,\cO_S(h\mp \eta)\big)=0$. In particular $\cO_S(h)$ is non--special. 
Finally $hK_S<0$, hence $h^0\big(S,\cO_S(K_S\pm\eta)\big)=0$.

The statement then follows from Theorem \ref{tExistence} by taking any non--trivial $\cO_S(\eta)\in\Pic^0(S)$ of order $2$.
\end{proof} 

Recall that a {\sl geometrically ruled surface} is a surface $S$ with a surjective morphism $p\colon S\to E$ onto a smooth curve such that every fibre of $p$ is isomorphic to $\p1$. If $S$ is geometrically ruled, then $p_g(S)=0$ and $q(S)$ is the genus of $E$ (see \cite{Ha2}, Chapter V.2 for further details).

\begin{remark}
\label{rElliptic}
Let $S$ be a geometrically ruled surface on an elliptic curve $E$ so that $p_g(S)=0$ and $q(S)=1$. Thanks to the results in \cite{Ha2}, Chapter V.2, we know the existence of a vector bundle $\mathcal H$ of rank $2$ on $E$ such that $h^0\big(E,\mathcal H\big)\ne0$ and $h^0\big(E,\mathcal H(-P)\big)=0$ for each $P\in E$ and $S\cong\bP(\mathcal H)$. Then $p$ can be identified with the natural projection map $\bP(\mathcal H)\to E$. The group $\Pic(S)$ is generated by the class $\xi$ of $\cO_{\bP(\mathcal H)}(1)$ and by $p^*\Pic(E)$. If we set $\cO_E(\frak h):=\det(\mathcal H)$ and $e:=-\deg(\frak h)$, then $e\ge-1$ (see \cite{Na}). Moreover, $K_S=-2\xi+p^*\frak h$.

There exists an exact sequence
\begin{equation}
\label{seqElliptic}
0\longrightarrow\cO_E\longrightarrow\mathcal H\longrightarrow\cO_E(\frak h)\longrightarrow0.
\end{equation}
The symmetric product of Sequence \eqref{seqElliptic} yields 
\begin{equation}
\label{seqEllipticSymmetric}
0\longrightarrow\mathcal H(-\frak h)\longrightarrow S^2\mathcal H(-\frak h)\longrightarrow\cO_E(\frak h)\longrightarrow0.
\end{equation}

Sequence \eqref{seqElliptic} splits if and only if $\mathcal H$ is decomposable. Thus, if this occurs, then $S^2\mathcal H(-\frak h)$ contains $\cO_E$ as direct summand, whence
\begin{equation}
\label{Bigger}
h^0\big(S,\cO_S(-K_S)\big)\ge h^0\big(E,\cO_E\big)=1.
\end{equation}
because $h^0\big(S,\cO_S(-K_S)\big)=h^0\big(E,S^2\mathcal H(-\frak h)\big)$, thanks to the projection formula. 

Assume that $\mathcal H$ is indecomposable. Then either $\cO_E(\frak h)=\cO_E$ or $\cO_E(\frak h)\ne\cO_E$. In the first case the cohomology of Sequences \eqref{seqElliptic} and  \eqref{seqEllipticSymmetric} again implies Inequality \eqref{Bigger}.

If $\cO_E(\frak h)\ne\cO_E$, then Lemma 22 of \cite{At} implies that $S^2\mathcal H(-\frak h)$ is the direct sum of the three non--trivial elements of order $2$ of $\Pic(E)$, hence $h^0\big(S,\cO_S(-K_S)\big)=0$.

We conclude that a geometrically ruled surface on an elliptic curve is anticanonical if and only if $e\ge0$. 
\end{remark}

Thanks to the above remark and Corollary \ref{cAnticanonical}, we know that each geometrically ruled surface $S$ with $q(S)=1$ and $e\ge0$ supports special Ulrich bundles of rank $2$ with respect to each very ample line bundle $\cO_S(h)$. We can extend the result also to the case $e=-1$.

\begin{corollary}
\label{cElliptic}
Let $S$ be a geometrically ruled surface with $q(S)=1$ and endowed with a very ample line bundle $\cO_S(h)$.

Then $\cO_S(h)$ is non--special and $S$ supports special Ulrich bundles of rank $2$.
\end{corollary}
\begin{proof}
We have to prove the statement only for $e=-1$. If $\cO_S(h)=\cO_{\bP(\mathcal H)}(a\xi+p^*\frak b)$, then $\deg(\frak b)>-a/2$ (see \cite{Ha2}, Proposition V.2.21). Thus the Table in Proposition 3.1 of \cite{G--P} implies that $h^1\big(S,\cO_S(h\pm\eta)\big)=0$ for each $\eta\in\Pic^0(S)$.

Again the statement follows from Theorem \ref{tExistence} by taking any non--trivial $\cO_S(\eta)$ of order $2$.
\end{proof} 

\begin{remark}
The corollary above  extends Propositions 3.1, 3.3 and Theorem 3.4 of \cite{A--C--MR} to the range $e\le 0$, when $g=1$. 
\end{remark}

Recall that an embedded surface $S\subseteq\p N$ is called {\sl non--degenerate} if it is not contained in any hyperplane.

\begin{corollary}
\label{cFourspace}
Let $S\subseteq\p4$ be a non--degenerate non--special surface with $p_g(S)=0$.
Then $S$ supports special Ulrich bundles of rank $2$.
\end{corollary}
\begin{proof}
The cohomology of Sequence \eqref{seqHyperplane} tensored by $\cO_S(h)$ implies $h^1\big(C,i^*\cO_S(h)\big)=0$. In particular such surfaces are sectionally non--special (see \cite{I--M} for details). Non--special and sectionally non--special surfaces are completely classified in \cite{I--M} and \cite{M--R}. They satisfy $q(S)\le1$ and, if equality holds, then they are either quintic scrolls over elliptic curves, or the Serrano surfaces (these are very special bielliptic surfaces of degree $10$: see \cite{Se}). The results above and Section 4 of \cite{Cs} yields the statement.
\end{proof}

\begin{remark}
Linearly normal non--special surface $S\subseteq\p4$ with $p_g(S)=0$ satisfy $3\le h^2\le10$ (see \cite{I--M} and \cite{M--R}).
If $h^2\le6$, such surfaces are known to support Ulrich line bundles: see \cite{MR--PL1} for the case $q(S)=0$ and \cite{Bea3}, Assertion 2) of Proposition 5 for the case $q(S)=1$. 
\end{remark}

\section{Stability of Ulrich bundles}
\label{sStability}
We start this section by recalling the following result: see \cite{C--H2}, Theorem 2.9 for its proof.

\begin{theorem}
\label{tUnstable}
Let $X$ be a smooth variety endowed with a very ample line bundle $\cO_X(h)$.

If $\cE$ is an Ulrich bundle on $X$ with respect to $\cO_X(h)$, the following assertions hold:
\begin{enumerate}
\item $\cE$ is semistable and $\mu$--semistable;
\item $\cE$ is stable if and only if it is $\mu$--stable;
\item if
\begin{equation*}
\label{seqUnstable}
0\longrightarrow\mathcal L\longrightarrow\cE\longrightarrow\mathcal M\longrightarrow0
\end{equation*}
is an exact sequence of coherent sheaves with $\cM$ torsion free and $\mu(\mathcal L)=\mu(\cE)$, then both $\mathcal L$ and $\cM$ are Ulrich bundles.
\end{enumerate}
\end{theorem}

We now prove Theorem \ref{tStable} stated in the introduction.

\medbreak
\noindent{\it Proof of Theorem \ref{tStable}.}
Recall that $\cE$ is constructed as follows. First we choose  $C:=S\cap H\in\vert h\vert$ where $H\cong\p{N-1}$ is a general hyperplane: from now on  we denote by $i\colon C\to S$ the inclusion morphism. The Hilbert scheme $\mathcal H_C$ of $0$--dimensional subschemes of degree $N+1$ on $C$ has dimension $N+1$ and contains an open non--empty subset $\mathcal R\subseteq\mathcal H_C$ corresponding to reduced schemes of $N+1$ points in general position in $H$. If we choose  a general $Z\in\mathcal R$, then we finally construct $\cE$ from $Z$ by means of Theorem \ref{tCB}.

We now show that if $Z$ is very general inside $\mathcal R$, i.e. it is in the complement of a countable union of suitable proper closed subsets, then $\cE$ is stable.

To this purpose, let $\cO_S(D)$ be an Ulrich line bundle on $S$ (if any). By hypothesis $\pi(\cO_S(h))\ge2$, then
$$
(h+\eta-D)h=-\frac{h^2+hK_S}2=1-\pi(\cO_S(h))\le-1,
$$
hence
\begin{equation}
\label{eqVanishing}
h^0\big(S,\cI_{C\vert S}(2h+\eta-D)\big)=h^0\big(S,\cO_S(h+\eta-D)\big)=0,
\end{equation}
i.e. there are no divisors in $\vert 2h+\eta-D\vert$ containing $C$. 
Thus the cohomology of Sequence \eqref{seqHyperplane} tensored by $\cO_S(2h+\eta-D)$ yields the injectivity of the restriction map 
$$
h^0\big(S,\cO_S(2h+\eta-D)\big)\to H^0\big(C,i^*\cO_S(2h+\eta-D)\big).
$$
Since $(2h+\eta-D)h=N+1$, it follows that each $Z\subseteq A\in \vert 2h+\eta-D\vert$ representing a point the Hilbert scheme $\mathcal H_C$ of subschemes of degree $N+1$ on $C$, is actually cut out on $C$ by $A$. 

Thus, if $\mathcal Z_D$ denotes the closed subset of $\mathcal H_C$ of points $Z$ such that $h^0\big(C,\cI_{Z\vert C}\otimes i^*\cO_S(2h+\eta-D)\big)\ge1$, then
$$
\dim(\mathcal Z_D)= h^0\big(C,i^*\cO_S(2h+\eta-D)\big)-1.
$$

On the one hand, if $i^*\cO_S(2h+\eta-D)$ is special, then the Clifford theorem and the second Equality \eqref{eqLineBundle} imply
\begin{equation*}
\label{BoundSpecial}
h^0\big(C,i^*\cO_S(2h+\eta-D)\big)\le\frac{(2h+\eta-D)h}2+1=\frac{N+3}2\le N,
\end{equation*}
because $N\ge4$ (see Inequality \eqref{eqDimension}). On the other hand, if $i^*\cO_S(2h+\eta-D)$ is non--special, the Riemann--Roch theorem on $C$ and the second Equality \eqref{eqLineBundle} return
\begin{equation*}
\label{BoundNonSpecial}
h^0\big(C,i^*\cO_S(2h+\eta-D)\big)=N+2-\pi(\cO_S(h))\le N,
\end{equation*}
because $\pi(\cO_S(h))\ge2$. It follows from the above inequalities that $\dim(\mathcal Z_D)\le N-1$.

Since $q(S)=1$ and the N\'eron--Severi group of $S$ is a finitely generated abelian group, it follows that the set $\mathcal D\subseteq\Pic(S)$ of Ulrich line bundles is contained in a countable disjoint union of a fixed elliptic curve. In particular there is
$$
Z\in \mathcal R\setminus \bigcup_{\cO_S(D)\in\mathcal D}\mathcal Z_D
$$
because $\dim(\mathcal R)=N+1$. Let $\cE$ be the corresponding bundle.

Assume that $\cE$ is not stable: then it is not $\mu$--stable, thanks to Theorem \ref{tUnstable}. In particular there exists a line subbundle $\cO_S(D)\subseteq\cE$ such that $\mu(\cE)=\mu(\cO_S(D))$. Again Theorem \ref{tUnstable} implies that  $\cO_S(D)$ is Ulrich. 

On the one hand, if $\cO_S(D)$ is contained in the kernel $\cK\cong\cO_S(h+K_S+\eta)$ of the map $\cE\to\cI_{Z\vert S}(2h+\eta)$ in Sequence \eqref{seqUlrich}, then $h^0\big(S,\cO_S(h+K_S+\eta-D)\big)\ne0$. On the other hand, Equality \eqref{eqLineBundle} and Inequality \eqref{eqDimension} imply
$$
(h+K_S+\eta-D)h=-\frac{h^2-hK_S}2=1-N\le-3,
$$
whence $h^0\big(S,\cO_S(h+K_S+\eta-D)\big)=0$. 

We deduce that $\cO_S(D)\not\subseteq\cK$, hence the composite map $\cO_S(D)\subseteq\cE\to\cI_{Z\vert S}(2h+\eta)$ should be non--zero, i.e.
$h^0\big(S,\cI_{Z\vert S}(2h+\eta-D)\big)\ge1$.
The cohomology of Sequence \eqref{seqIdeals} tensored by $\cO_S(2h+\eta-D)$ and Equality \eqref{eqVanishing} then would imply 
$$
h^0\big(C,\cI_{Z\vert C}\otimes\cO_S(2h+\eta-D)\big)\ge h^0\big(S,\cI_{Z\vert S}(2h+\eta-D)\big)\ge1,
$$
contradicting our choice of $Z$: thus the bundle $\cE$ is necessarily stable.
\qed
\medbreak

\begin{remark}
If $\pi(\cO_S(h))=1$, then $S$ is a geometrically ruled surface embedded as a scroll by $\cO_S(h)\cong\cO_S(\xi+p^*\frak b)$ over an elliptic curve, thanks to \cite{A--S}, Theorem A (here we are using the notation introduced in Remark \ref{rElliptic}). 

Moreover $(h+\eta-D)h=0$, hence the argument in the above proof does not lead to any contradiction when $\cO_S(D)\cong\cO_S(h+\eta)$. Such a line bundle is actually  Ulrich, because one easily checks that it satisfies all the conditions of Corollary \ref{cUlrichLine}.

In \cite{Work}, via a slightly different but similar construction, we are able to show the existence special stable Ulrich bundles of rank $2$ on  elliptic scrolls.
\end{remark}

Let $S$ be a surface with $p_g(S)=0$, $q(S)=1$ and endowed with a very ample non--special line bundle $\cO_S(h)$. Let  
\begin{equation*}
\label{eqUlrichSpecial}
c_1:=3h+K_S+2\eta,\qquad c_2:=\frac{5h^2+3hK_S}2,
\end{equation*}
where $\cO_S(\eta)\in\Pic^0(S)\setminus\{\ \cO_S\ \}$ satisfies
$$
h^0\big(S,\cO_S(K_S\pm\eta)\big)=h^1\big(S,\cO_S(h\pm\eta)\big)=0.
$$
If $\pi(\cO_S(h))\ge2$, then the coarse moduli space $\cM_S^{s}(2;c_1,c_2)$ parameterizing isomorphism classes of stable rank $2$ bundles on $S$ with Chern classes $c_1$ and $c_2$ is non--empty (see Theorem \ref{tStable}). 
The locus $\cM_S^{s,U}(2;c_1,c_2)\subseteq \cM_S^{s}(2;c_1,c_2)$ parameterizing stable Ulrich bundles is open as pointed out in \cite{C--H2}.

\begin{proposition}
\label{pComponent}
Let $S$ be a surface with $p_g(S)=0$, $q(S)=1$ and endowed with a very ample non--special line bundle $\cO_S(h)$.

If $\pi(\cO_S(h))\ge2$, then there is a component $\cU_S(\eta)$ of dimension at least $h^2-K_S^2$ in $\cM_S^{s,U}(2;c_1,c_2)$ containing all the points representing the stable bundles $\cE$ constructed in Theorem \ref{tExistence}.
\end{proposition}
\begin{proof}
Let us denote by $\mathcal H_S$ the Hilbert flag scheme of pairs $(Z,C)$ where $C\in \vert \cO_S(h)\vert$ and $Z\subseteq C$ is a $0$--dimensional subscheme of degree $N+1$. The general $C\in \vert \cO_S(h)\vert$ is smooth and its image via the map induced by $\cO_S(h)$ generate a hyperplane inside $\p N$. Thus the set $\mathcal H_S^U\subseteq \mathcal H_S$ of pairs $(Z,C)$ corresponding to sets of points $Z$ in a smooth curve $C\subseteq\p N$ which are in general position in the linear space generated by $C$ is open and non--empty. 

We have a well--defined forgetful dominant morphism $\mathcal H_S\to \vert \cO_S(h)\vert$ whose fibre over $C$ is an open subset of the $(N+1)$--symmetric product of $C$. In particular $\mathcal H_S^U$ is irreducible of dimension $2N+1$. Let $(Z,C)$ represent a point of $\mathcal H_S^U$: the Ulrich bundles associated to such a point via the construction described in Theorem \ref{tExistence} correspond to the sections of
$$
\mathrm{Ext}^1_S\big(\cI_{Z\vert S}(h-K_S),\cO\big)\cong H^1\big(S,\cI_{Z\vert S}(h)\big)^\vee.
$$
By definition of $\mathcal H_S^U$, we have $h^0\big(C,\cI_{Z\vert C}(h)\big)=0$, hence the cohomology of the exact sequence
$$
0\longrightarrow \cI_{Z\vert C}(h)\longrightarrow  \cO_C(h)\longrightarrow  \cO_Z(h)\longrightarrow 0
$$
and the Riemann--Roch theorem for $\cO_C(h)$ yield $h^1\big(C,\cI_{Z\vert C}(h)\big)=\deg(Z)-\chi(\cO_C(h))=0$. Sequence \eqref{seqIdeals},  the isomorphism $\cI_{C\vert S}\cong\cO_S(-h)$ and the hypothesis $q(S)=p_g(S)=0$ finally return $h^1\big(S,\cI_{Z\vert S}(h)\big)=1$. Thus we have a family $\frak E$ of Ulrich bundles of rank $2$ with Chern classes $c_1$ and $c_2$ parameterized by $\mathcal H_S^U$. 

If $\pi(\cO_S(h))\ge2$, then the bundles in the family are also stable for a general choice of $Z$. Since stability is an open property in a flat family (see \cite{H--L}, Proposition 2.3.1 and Corollary 1.5.11), it follows the existence of an irreducible open subset $\mathcal H_S^{s,U}\subseteq \mathcal H_S^U\subseteq \mathcal H_S$  of points  corresponding to stable bundles. 

Thus, we have a morphism $\mathcal H_S^{s,U}\to \cM_S^{s,U}(2;c_1,c_2)$ whose image parameterizes the isomorphism classes of stable bundles constructed in Theorem \ref{tExistence}. In particular such bundles, correspond to the points of a single irreducible component $\cU_S(\eta)\subseteq\cM_S^{s,U}(2;c_1,c_2)$.

Theorems 4.5.4 and 4.5.8 of \cite{H--L} imply that $\dim(\cU_S(\eta))\ge4c_2-c_1^2-3\chi(\cO_S)$. Taking into account the definitions of $c_1$ and $c_2$, simple computations finally yield $\dim(\cU_S(\eta))\ge h^2-K_S^2$.
\end{proof}

If we have some extra informations on the surface $S$, then we can describe $\cU_S(\eta)$ as the following proposition shows.

\begin{proposition}
\label{pAnticanonicalStable}
Let $S$ be an anticanonical surface with $p_g(S)=0$, $q(S)=1$ and endowed with a very ample line bundle $\cO_S(h)$.

If $\pi(\cO_S(h))\ge2$, then $\cU_S(\eta)$ is non--rational and generically smooth of dimension $h^2-K_S^2$.
\end{proposition}
\begin{proof}
Thanks to Corollary \ref{cAnticanonical} we know that $\cO_S(h)$ is non--special. Let $A\in\vert-K_S\vert$: the cohomology of 
$$
0\longrightarrow\cO_S(K_S)\longrightarrow\cO_S\longrightarrow\cO_A\longrightarrow0
$$
tensored with $\cE\otimes\cE^\vee$ yields the exact sequence
$$
0\longrightarrow H^0\big(S,\cE\otimes\cE^\vee(K_S)\big)\longrightarrow H^0\big(S,\cE\otimes\cE^\vee\big)\longrightarrow H^0\big(A,\cE\otimes\cE^\vee\otimes\cO_A\big).
$$
Since $\cE$ is stable (see Theorem \ref{tStable}), then it is simple, i.e. $h^0\big(S,\cE\otimes\cE^\vee\big)=1$ (see \cite{H--L}, Corollary 1.2.8), hence the map
$$
H^0\big(S,\cE\otimes\cE^\vee\big)\longrightarrow H^0\big(A,\cE\otimes\cE^\vee\otimes\cO_A\big)
$$
is injective. We deduce that $h^2\big(S,\cE\otimes\cE^\vee\big)=h^0\big(S,\cE\otimes\cE^\vee(K_S)\big)=0$.

Thus $\cE$ corresponds to a smooth point of $\cU_S(\eta)$ and $\dim(\cU_S(\eta))=h^2-K_S^2$,  thanks to Corollary 4.5.2 of \cite{H--L}. Finally, being $q(S)=1$, then $\cU_S(\eta)$ is irregular thanks to \cite{B--C} as well.
\end{proof}

Remark \ref{rElliptic} and the above proposition yield the following corollary.

\begin{corollary}
\label{cEllipticStable}
Let $S$ be a geometrically ruled surface with $q(S)=1$, $e\ge0$ and endowed with a very ample line bundle $\cO_S(h)$.

If $\pi(\cO_S(h))\ge2$, then $\cU_S(\eta)$ is non--rational and generically smooth of dimension $h^2$.
\end{corollary}

\section{Ulrich--wildness}
\label{sWild}
Let $S$ be a surface with $p_g(S)=0$ and $q(S)=1$. Moreover $\pi(\cO_S(h))\ge 1$ because $S$ is not rational, as pointed out in the introduction.

We will make use of the following result.

\begin{theorem}
\label{tFPL}
Let $X$ be a smooth variety endowed with a very ample line bundle $\cO_X(h)$.

If $\cA$ and $\cB$ are simple Ulrich bundles on $X$ such that $h^1\big(X,\cA\otimes\cB^\vee\big)\ge3$ and $h^0\big(X,\cA\otimes\cB^\vee\big)=h^0\big(X,\cB\otimes\cA^\vee\big)=0$, then $X$ is Ulrich--wild.
\end{theorem}
\begin{proof}
See \cite{F--PL}, Theorem 1 and Corollary 1.
\end{proof}

An immediate consequence of the above Theorem is the proof of Theorem \ref{tWild}.

\medbreak
\noindent{\it Proof of Theorem \ref{tWild}.}
Recall that $S$ is a surface with $p_g(S)=0$, $q(S)=1$ and endowed with a very ample non--special line bundle $\cO_S(h)$. We have $\pi(\cO_S(h))\ge1$, $\chi(\cO_S)=0$ and $K_S^2\le0$.

If $\pi(\cO_S(h))\ge2$, then Theorems \ref{tExistence} and \ref{tStable} yield the existence of a stable special Ulrich bundle $\cE$ of rank $2$ on $S$. 

The local dimension of $\cM_S^{s}(2;c_1,c_2)$ at the point corresponding to $\cE$ is at least $4c_2-c_1^2=h^2-K_S^2\ge1$. Thus, there exists a second stable Ulrich bundle $\cG\not\cong\cE$ of rank $2$ with $c_i(\cG)=c_i$, for $i=1,2$. Both $\cE$ and $\cG$, being stable, are simple (see \cite{H--L}, Corollary 1.2.8).

Due to Proposition 1.2.7 of \cite{H--L} we have $h^0\big(F,\cE\otimes\cG^\vee\big)=h^0\big(F,\cG\otimes\cE^\vee\big)=0$, thus
$$
h^1\big(F,\cE\otimes\cG^\vee\big)=h^2\big(F,\cE\otimes\cG^\vee\big)-\chi(\cE\otimes\cG^\vee)\ge -\chi(\cE\otimes\cG^\vee).
$$
Equality \eqref{RRGeneral} with $\cF:=\cE\otimes\cG^\vee$ and the equalities $\rk(\cE\otimes\cG^\vee)=4$, $c_1(\cE\otimes\cG^\vee)=0$ and $c_2(\cE\otimes\cG^\vee)=4c_2-c_1^2$ imply
$$
h^1\big(F,\cE\otimes\cG^\vee\big)\ge 4c_2-c_1^2=h^2-K_S^2\ge3.
$$
because surfaces of degree up to $2$ are rational. We conclude that $S$ is Ulrich--wild, by Theorem \ref{tFPL}.

Finally let $\pi(\cO_S(h))=1$. In this case, $S$ is a geometrically ruled surface on an elliptic curve $E$ thanks to Theorem A of \cite{A--S} embedded as a scroll bay $\cO_S(h)$. Using the notations of Remark \ref{rElliptic}  we can thus assume that $\cO_S(h)=\cO_S(\xi+p^*\frak{b})$, where $\deg(\frak b)\ge e+3$.

Assertion 2) of Proposition 5 in \cite{Bea3} yields that for each $\vartheta\in\Pic^0(E)\setminus\{\ \cO_E\ \}$ the line bundle $\mathcal L:=\cO_S(h+p^*\vartheta)\cong \cO_S(\xi+p^*\frak{b}+p^*\vartheta)$ is Ulrich. It follows from Corollary \ref{cUlrichLine} that $\cM:=\cO_S(2h+K_S-p^*\vartheta)\cong p^*\cO_E(2\frak b+\frak h-\vartheta)$ is Ulrich too. 

Trivially, such bundles are simple and $h^0\big(S,\mathcal L\otimes\cM^\vee\big)=h^0\big(S,\cM\otimes\mathcal L^\vee\big)=0$ because $\mathcal L\not\cong\cM$. Since $\mathcal L\otimes\cM^\vee\cong\cO_S(\xi-p^*\frak b-p^*\frak h+2\vartheta)$ and $e=-\deg(\frak h)\ge-1$, it follows from Equality \eqref{RRGeneral} that
$$
h^1\big(S,\mathcal L\otimes\cM^\vee\big)\ge-\chi(\mathcal L\otimes\cM^\vee)=2\deg(\frak b)-e\ge e+6\ge5.
$$
The statement thus again follows from Theorem \ref{tFPL}.
\qed
\medbreak

The following consequence of the above theorem is immediate, thanks to Corollaries \ref{cBielliptic}, \ref{cAnticanonical}, \ref{cElliptic}.

\begin{corollary}
Let $S$ be a surface endowed with a very ample line bundle $\cO_S(h)$.

If $S$ is either bielliptic, or anticanonical with $q(S)=1$, or geometrically ruled with $q(S)=1$, then it is Ulrich--wild.
\end{corollary}

The following corollary strengthens the second part of the statements of Theorems 4.13 and 4.18 in \cite{MR--PL1}.

\begin{corollary}
Let $S\subseteq\p4$ be a non--degenerate linearly normal non--special surface of degree at least $4$ with $p_g(S)=0$.
Then $S$ is Ulrich--wild.
\end{corollary}
\begin{proof}
As pointed out in the proof of Corollary \ref{cFourspace} the surface $S$ satisfies $q(S)\le1$ and if equality holds it is either an elliptic scroll or a bielliptic surface. Theorem \ref{tWild} above and Section 5 of \cite{Cs} yields that $S$ is Ulrich--wild.
\end{proof}

\bigskip
\noindent
Gianfranco Casnati,\\
Dipartimento di Scienze Matematiche, Politecnico di Torino,\\
c.so Duca degli Abruzzi 24, 10129 Torino, Italy\\
e-mail: {\tt gianfranco.casnati@polito.it}

\end{document}